\documentclass[12pt]{amsart}

\usepackage{amssymb}

\newtheorem{theorem}{Theorem}[section]
\newtheorem{lemma}[theorem]{Lemma}
\newtheorem{proposition}[theorem]{Proposition}
\newtheorem{definition}[theorem]{Definition}
\newtheorem{remark}[theorem]{Remark}

\numberwithin{equation}{section}

\begin{document}
\baselineskip=15.5pt

\title{Homogeneous principal bundles and stability}

\author[I. Biswas]{Indranil Biswas}

\address{School of Mathematics, Tata Institute of Fundamental
Research, Homi Bhabha Road, Bombay 400005, India}

\email{indranil@math.tifr.res.in}


\date{}

\begin{abstract}
Let $G/P$ be a rational homogeneous
variety, where $P$ is a parabolic subgroup of a
simple and simply connected linear algebraic group
$G$ defined over an algebraically closed field of
characteristic zero. A homogeneous principal bundle
over $G/P$ is semistable (respectively,
polystable) if and only if it is equivariantly semistable
(respectively, equivariantly polystable). A stable
homogeneous principal $H$--bundle $(E_H\, ,\rho)$ is
equivariantly stable, but the converse is not true
in general. If a homogeneous principal $H$--bundle
$(E_H\, ,\rho)$ is equivariantly stable, but $E_H$
is not stable,
then the principal $H$--bundle $E_H$ admits an
action $\rho'$ of $G$ such that the pair
$(E_H\, ,\rho')$ is a homogeneous principal
$H$--bundle which is not equivariantly stable.

\end{abstract}

\maketitle

\section{Introduction}\label{sec1}

Let $G$ be a simple and simply connected linear algebraic group
defined over an algebraically closed field $k$ of characteristic
zero. Fix a proper parabolic subgroup $P$ of $G$. Fix a very
ample line bundle $\xi$ on the projective variety $G/P$.
Let $H$ be any reductive linear algebraic group defined over $k$.
For any homomorphism
$$
\eta\, :\, P\, \longrightarrow\, H
$$
with the property that the image of $\eta$ is not contained
in any proper parabolic subgroup of $H$, the associated
principal $H$--bundle $G\times^P H$ over $G/P$ is known to
be stable with respect to $\xi$
\cite[page 576, Theorem 2.6]{AzB}.

A homogeneous principal $H$--bundle on $G/P$ is a principal
$H$--bundle $E_H\, \longrightarrow\, G/P$ together with an
action of $G$ 
$$
\rho\, :\, G\times E_H\, \longrightarrow\, E_H
$$
that lifts the left--translation action of $G$ on $G/P$. 
It may be mentioned that all homogeneous principal $H$--bundles
over $G/P$ are given by homomorphisms from $P$ to $H$.
Here we consider those homogeneous principal $H$--bundles over
$G/P$ that arise from homomorphisms for which the image
is contained in some proper parabolic subgroup of $H$.
We also consider a weaker notion of stability. A homogeneous
principal $H$--bundle $(E_H\, ,\rho)$ over $G/P$ is called
equivariantly stable (respectively, equivariantly semistable) if
the usual stability condition (respectively, the semistability
condition) holds for those reduction of structure groups of
$E_H$ that are preserved by the action $\rho$ of $G$ on $E_H$.
Equivariantly polystable homogeneous principal $H$--bundles
are defined similarly.

We show that a homogeneous principal $H$--bundle
$(E_H\, ,\rho)$ over $G/P$
is equivariantly semistable if and only if the principal
$H$--bundle $E_H$ is semistable (Lemma \ref{lem1}). Similarly,
$(E_H\, ,\rho)$ is equivariantly polystable if and only if the
principal $H$--bundle $E_H$ is polystable (Lemma \ref{lem2}).

If $E_H$ is stable, then $(E_H\, ,\rho)$ is equivariantly
stable. But the converse is not true. However the following
weak converse holds (see Theorem \ref{thm1}):

\begin{theorem}\label{thm0}
Let $(E_H\, ,\rho)$ be an equivariantly stable homogeneous
principal $H$--bundle over $G/P$ such that the principal
$H$--bundle $E_H$ is not stable. Then there is an action
of $G$ on $E_H$
$$
\rho'\, :\, G \times E_H\, \longrightarrow\, E_H
$$
such that the following two hold:
\begin{enumerate}
\item the pair $(E_H\, ,\rho')$ is a homogeneous
principal $H$--bundle, and

\item the homogeneous principal $H$--bundle
$(E_H\, ,\rho')$ is not equivariantly stable.
\end{enumerate}
\end{theorem}

\noindent
\textbf{Acknowledgements.}\, The author is very grateful to
the referee for providing comments to improve the paper.

\section{Preliminaries}\label{sec2}

Let $k$ be an algebraically closed field of characteristic zero.
Let $G$ be a simple and simply connected linear algebraic group
defined over the field $k$. Fix a proper parabolic subgroup 
$$
P\,\subset\, G\, .
$$
So the quotient
\begin{equation}\label{e1}
M\, :=\, G/P
\end{equation}
is an irreducible smooth projective variety defined over $k$.
The quotient map
\begin{equation}\label{e2}
f_0\, :\, G\, \longrightarrow\, G/P
\end{equation}
defines a principal $P$--bundle over $M$. The left translation
action of $G$ on itself defines a homomorphism
\begin{equation}\label{e3}
\phi\, :\, G\, \longrightarrow\, \text{Aut}(M)\, .
\end{equation}

Fix a very ample line bundle $\xi$ on $M$. It is known that
any ample line bundle on $M$ is very ample. The degree of any
torsionfree coherent sheaf on $M$ will be defined using $\xi$.
More precisely, for any torsionfree coherent sheaf $F$ on $M$,
the degree of $F$ is defined to be the degree of the
restriction of $F$ to any smooth complete intersection curve 
on $M$ obtained by intersecting hyperplanes from the complete
linear system $\vert\xi\vert$. Let $F$ be a
vector bundle defined over a nonempty Zariski open dense
subset $U\, \subseteq\, G/P$ such that the
codimension of the complement
$(G/P)\setminus U$ is at least two. Then
the direct image $\iota_* F$ is a torsionfree
coherent sheaf on $G/P$, where $\iota\, :\, U\,
\hookrightarrow\, G/P$ is the inclusion map. For
such a coherent sheaf $F$ define
$$
\text{degree}(F)\, :=\, \text{degree}(\iota_* F)\, .
$$

Let $H$ be a connected reductive linear algebraic group
defined over the field $k$. Let $Q$ be a proper
parabolic subgroup of $H$, and let $\lambda$ be a
character of $Q$ which is trivial on the connected
component of the center of $H$ containing the identity
element. Such a character $\lambda$ is called
\textit{strictly anti--dominant} if the associated line
bundle $L_\lambda \, =\, H\times^{Q} k$ over $H/Q$ is ample.

A principal $H$--bundle $E_H$ over
$M$ is called \textit{stable} (respectively,
\textit{semistable}) if for every triple of the form
$(Q\, , E_Q\, , \lambda)$, where
\begin{itemize}
\item $Q\, \subsetneq\, H$ is a proper parabolic
subgroup, and
\begin{equation}\label{d.s}
E_Q\, \subset\, E_H
\end{equation}
is a reduction of structure
group of $E_H$ to $Q$ over some nonempty Zariski open
subset $U\, \subset\, G/P$ such
that the codimension of the complement $(G/P)\setminus U$
is at least two, and

\item $\lambda$ is some strictly anti--dominant character of
$Q$ (see the above definition of an anti--dominant character),
\end{itemize}
the inequality
$$
\text{degree}(E_Q(\lambda)) \, >\, 0
$$
(respectively, $\text{degree}(E_Q(\lambda)) \, \geq\, 0$) holds,
where $E_Q(\lambda)$ is the line bundle over $U$ associated to
the principal $Q$--bundle $E_Q$ for the character $\lambda$ of $Q$.

In order to be able to
decide whether a given principal $H$--bundle $E_H$ is
stable (respectively, semistable), it suffices to verify the
above strict inequality (respectively, inequality) only for
the maximal proper
parabolic subgroups of $H$. More precisely,
$E_H$ is stable (respectively, semistable) if and
only if for every pair of the form $(Q\, , \sigma)$, where
\begin{itemize}
\item $Q\, \subset\, H$ is a proper maximal parabolic subgroup,
and 
\item $\sigma$ is a reduction of structure group of $E_H$ to $Q$
\begin{equation}\label{sigma}
\sigma\, :\, U\, \longrightarrow\, E_H/Q
\end{equation}
over some Zariski open dense subset $U\, \subset\, G/P$ such
that the codimension of the complement $(G/P)\setminus U$
is at least two,
\end{itemize}
the inequality
\begin{equation}\label{st.eq}
\text{degree}(\sigma^* T_{\text{rel}})\, >\, 0
\end{equation}
\begin{equation}\label{st.eq1}
{\rm (respectively,\,~\,}~\, ~\,
\text{degree}(\sigma^* T_{\text{rel}})\,
\geq\, 0{\rm )}
\end{equation}
holds, where $T_{\text{rel}}$ is the relative
tangent bundle over $E_H/Q$ for the natural projection
$E_H/Q\, \longrightarrow\, G/P$. (See
\cite[page 129, Definition 1.1]{Ra} and \cite[page 131,
Lemma 2.1]{Ra}.)

Let $E_H$ be a principal $H$--bundle over $G/P$. A reduction
of structure group of $E_H$
\[
E_Q\, \subset\, E_H
\]
to some parabolic subgroup $Q\, \subset\, H$ is called
\textit{admissible} if for each character $\lambda$ of
$Q$ trivial on the center of $H$, the degree of the
associated line bundle $E_Q(\lambda)\, =\,
{E_Q}\times^Q k$ is zero
\cite[page 307, Definition 3.3]{Ra}.

The unipotent radical of a parabolic subgroup $Q\, \subset
\,H$ will be denoted by $R_u(Q)$. The quotient group
\[
L(Q) \, :=\, P/R_u(Q)\, ,
\]
which is called the \textit{Levi quotient} of $Q$, is a
connected reductive linear algebraic group defined over $k$.
A \textit{Levi subgroup} of $Q$ is a closed connected
reductive subgroup
$$
L'\, \subset\, Q
$$
such that the composition homomorphism
$$
L'\, \hookrightarrow\, Q\, \longrightarrow\, L(Q)
$$
is an isomorphism (here $Q\, \longrightarrow\, L(Q)$ is the
quotient map). (See \cite[page 158, \S~11.22]{Bo} and
\cite[page 184, \S~30.2]{Hu}.) The notation $L(Q)$
will also be used for denoting a Levi subgroup of $Q$.

A principal $H$--bundle $E_H$ over $G/P$ is called
\textit{polystable} if either $E_H$ is stable, or there
is a proper parabolic subgroup $Q\, \subset\, H$
and a reduction of structure group over $G/P$
$$
E_{L(Q)}\, \subset\, E_H
$$
to a Levi subgroup $L(Q)$ of $Q$ such that the
following two conditions hold:
\begin{enumerate}
\item the principal $L(Q)$--bundle $E_{L(Q)}$ is stable, and 

\item the reduction of structure group of $E_H$ to $Q$
obtained by extending the structure group of $E_{L(Q)}$
using the inclusion of $L(Q)$ in $Q$ is admissible.
\end{enumerate}

Let $H'$ be any linear algebraic group defined over $k$.

\begin{definition}\label{def1}
{\rm A} homogeneous {\rm principal $H'$--bundle over $G/P$
is a principal $H'$--bundle}
\begin{equation}\label{e4}
f\, :\, E_{H'}\, \longrightarrow\, G/P
\end{equation}
{\rm together with an action of $G$}
\begin{equation}\label{rho}
\rho\, :\, G \times E_{H'}\, \longrightarrow\, E_{H'}
\end{equation}
{\rm such that the following two conditions hold:}
\begin{enumerate}
\item $f\circ \rho (g\, ,z)\, =\, \phi(g)(f(z))$
{\rm for all $(g\, ,z)\, \in\, G
\times E_{H'}$, where $\phi$ and $f$ are defined in
Eq. \eqref{e3} and Eq. \eqref{e4} respectively, and}

\item {\rm the actions of $G$ and $H'$ on $E_{H'}$ commute.}
\end{enumerate}
\end{definition}

Let $H$ be a connected reductive linear algebraic group
defined over $k$.

\begin{definition}\label{def2}
{\rm A homogeneous principal $H$--bundle
$(E_H\, ,\rho)$ is called} equivariantly
stable {\rm (respectively,} equivariantly semistable{\rm{)}
if the condition in the definition of stability
(respectively, semistability) holds for all $E_Q$ as
in Eq. \eqref{d.s} that are left invariant by the action
$\rho$ of $G$ on $E_H$.}

{\rm Similarly, a homogeneous
principal $H$--bundle $(E_H\, ,\rho)$ is called}
equivariantly polystable {\rm if either $E_H$ is
equivariantly stable, or there
is a proper parabolic subgroup $Q\, \subset\, H$
and a reduction of structure group over $G/P$
$$
E_{L(Q)}\, \subset\, E_H
$$
to a Levi subgroup $L(Q)$ of $Q$ such that the
following three conditions hold:}
\begin{enumerate}
\item {\rm the action of $G$ on $E_H$ leaves $E_{L(Q)}$
invariant,}

\item {\rm the principal $L(Q)$--bundle $E_{L(Q)}$ is
equivariantly stable, and}

\item {\rm the reduction of structure group of $E_H$ to $Q$
obtained by extending the structure group of $E_{L(Q)}$
using the inclusion of $L(Q)$ in $Q$ is admissible.}
\end{enumerate}
\end{definition}

\begin{remark}
{\rm A homogeneous principal $H$--bundle
$(E_H\, ,\rho)$ is equivariantly
stable if the inequality in Eq. \eqref{st.eq} holds
for all $\sigma$ as
in Eq. \eqref{sigma} that are invariant under the action
of $G$ on $E_H/Q$ defined by $\rho$.}

{\rm Similarly, a homogeneous principal $H$--bundle
$(E_H\, ,\rho)$ is equivariantly
semistable if the inequality in Eq. \eqref{st.eq1} holds
for all $\sigma$ as
in Eq. \eqref{sigma} that are invariant under the action
of $G$ on $E_H/Q$ defined by $\rho$.}
\end{remark}

\section{A criterion for homogeneous principal bundles}

If $(E_{H'}\, ,\rho)$ is a homogeneous
principal $H'$--bundle over $G/P$,
then for each point $g\, \in\, G$, the pulled back principal
$H'$--bundle $\phi(g)^*E_{H'}$ is isomorphic to $E_{H'}$,
where $\phi$ is the homomorphism in Eq. \eqref{e3}. Indeed,
the automorphism of the variety $E_{H'}$ defined by 
$z\, \longmapsto\, \rho(g\, ,z)$ gives an isomorphism
of the principal $H'$--bundles $E_{H'}\, \longrightarrow\,
\phi(g)^*E_{H'}$.

The following proposition asserts a converse of the above
observation.

\begin{proposition}\label{prop1}
Let $H'$ be a linear algebraic group defined over $k$. Let
$$
\gamma\, :\, E_{H'}\, \longrightarrow\, G/P
$$
be a principal $H'$--bundle such that
for each point $g\, \in\, G$, the pulled back principal
$G$--bundle $\phi(g)^*E_{H'}$ is isomorphic to $E_{H'}$,
where $\phi$ is the homomorphism in Eq. \eqref{e3}. Then
there is an action of $G$ on $E_{H'}$
$$
\rho\, :\, G \times E_{H'}\, \longrightarrow\, E_{H'}
$$
such that the pair $(E_{H'}\, ,\rho)$
is a homogeneous principal $H'$--bundle.
\end{proposition}

\begin{proof}
Let $\mathcal A$ denote the group of automorphisms of
the principal $H'$--bundle $E_{H'}$. So $\mathcal A$
consists of all automorphisms of the variety $E_{H'}$
$$
h\, :\, E_{H'}\, \longrightarrow\, E_{H'}
$$
such that
\begin{itemize}
\item $\gamma\circ h\, =\, \gamma$, and
\item $h$ commutes with the action of $H'$ on $E_{H'}$.
\end{itemize}

We will show that
$\mathcal A$ is a linear algebraic group defined over $k$.

Fix a finite dimensional faithful representation
\begin{equation}\label{e5}
\tau\,:\, H'\, \longrightarrow\, \text{GL}(V)
\end{equation}
of $H'$.
Let $E_V\, :=\, E_{H'}\times^{H'} V$ be the vector bundle
over $G/P$ associated to $E_{H'}$ for this $H'$--module $V$.
Let $\text{Aut}(E_V)$ denote the group of all automorphisms
of the vector bundle $E_V$. We note that
$$
\text{Aut}(E_V)\, \hookrightarrow\, H^0(G/P,\, E_V\otimes
E^*_V)\, .
$$
Using this inclusion, $\text{Aut}(E_V)$ has the structure
of a linear algebraic group defined over $k$. Any automorphism
of the principal $H'$--bundle $E'_H$ yields an automorphism
of the associated vector bundle $E_V$. Since $\tau$ in
Eq. \eqref{e5} is injective, it follows that $\mathcal A$
is a closed subgroup of $\text{Aut}(E_V)$. Hence $\mathcal A$
is a linear algebraic group defined over $k$.

Let $\widetilde{\mathcal A}$ denote the group of all
pairs of the form $(g\, ,h)$, where $g\, \in\, G$, and
$$
h\, :\, E_{H'}\, \longrightarrow\, E_{H'}
$$
is an automorphism of the variety $E_{H'}$ satisfying the
following two conditions:
\begin{enumerate}
\item $\gamma\circ h\, =\,\phi(g)\circ\gamma$, and

\item $h$ commutes with the action of $H'$ on $E_{H'}$.
\end{enumerate}
The group operation on $\widetilde{\mathcal A}$ is:
$$
(g_1\, ,h_1)(g_2\, ,h_2)\, :=\,(g_1\circ g_2\, , h_1\circ h_2)\, .
$$
We will show that $\widetilde{\mathcal A}$ is also a linear
algebraic group defined over $k$.

Let
\begin{equation}\label{e6}
\phi_0\, :\, G\times M\, \longrightarrow\, M\, :=\, G/P
\end{equation}
be the left action defined by $\phi$ in Eq. \eqref{e3}.
Let
\begin{equation}\label{e7}
p_2\, :\, G\times M\, \longrightarrow\, M
\end{equation}
be the projection to the second factor. Let $\mathcal S$ denote
the sheaf of isomorphisms from the principal $H'$--bundle
$p^*_2 E_{H'}$ to the principal $H'$--bundle $\phi^*_0 E_{H'}$
over $G\times M$, where $\phi_0$ and $p_2$ are defined in
Eq. \eqref{e6} and Eq. \eqref{e7} respectively. Now
consider the direct image
$$
\widetilde{\mathcal S} \, :=\, p_{1*}{\mathcal S}
$$
over $G$, where $p_1$ as before is the projection of $G\times
M$ to $G$. Comparing the definitions $\widetilde{\mathcal S}$
and $\widetilde{\mathcal A}$ it follows immediately
that $\widetilde{\mathcal S}$ is identified
with $\widetilde{\mathcal A}$.

As before, let $E_V$ denote the
vector bundle associated to $E_{H'}$ for the faithful $H'$--module
$V$ in Eq. \eqref{e5}. The total space of $\widetilde{\mathcal S}$
is naturally embedded in the total space of the vector bundle
\begin{equation}\label{bG}
p_{1*}((\phi^*_0 E_V)\otimes p^*_2 E^*_V)\, \longrightarrow\,
G\, ,
\end{equation}
where $\phi_0$ and $p_2$ are defined in Eq. \eqref{e6} and Eq.
\eqref{e7} respectively, and $p_1$ is the projection of
$G\times M$ to $G$. Using this embedding, the total space of
$\widetilde{\mathcal S}$ gets a structure of a scheme defined
over $k$. Consequently, the identification of
$\widetilde{\mathcal A}$
with $\widetilde{\mathcal S}$ makes $\widetilde{\mathcal A}$ a
scheme defined over $k$. The group operations (multiplication
and inverse maps) are algebraic. Hence $\widetilde{\mathcal A}$
is an algebraic group defined over $k$.

Since $G$ is an affine variety, the total space of the vector
bundle $p_{1*}((\phi^*_0 E_V)\otimes p^*_2 E^*_V)$ in
Eq. \eqref{bG} is also an affine variety. So
$\widetilde{\mathcal A}$ is an affine scheme. Therefore, we
conclude that $\widetilde{\mathcal A}$ is a linear
algebraic group defined over $k$.

Let
\begin{equation}\label{e8}
p\, :\, \widetilde{\mathcal A}\, \longrightarrow\, G
\end{equation}
be the homomorphism defined by $(g\, ,h)\, \longmapsto\, g$.
Let
\begin{equation}\label{e9}
I\, :\, {\mathcal A}\, \longrightarrow\, \widetilde{\mathcal A}
\end{equation}
be the homomorphism defined by $h\, \longmapsto\,
(e\, ,h)$, where $e\, \in\, G$ is the identity element.
Since $\phi(g)^*E_{H'}$ is isomorphic to $E_{H'}$ for all
$g\, \in\, G$, the homomorphism $p$ in Eq. \eqref{e8} is
surjective. Hence we have a short exact sequence of groups
\begin{equation}\label{ag}
e\,\longrightarrow\,{\mathcal A}\,
\stackrel{I}{\longrightarrow}\,\widetilde{\mathcal A}\,
\stackrel{p}{\longrightarrow}\,G\,\longrightarrow\,e\, ,
\end{equation}
where $I$ is defined in Eq. \eqref{e9}.

We will show that the short exact
sequence in Eq. \eqref{ag} is right split, or
in other words, there is a homomorphism
\begin{equation}\label{e10}
\psi\, :\,G\,\longrightarrow\, \widetilde{\mathcal A}
\end{equation}
such that $p\circ\psi\,=\, \text{Id}_G$.

To prove this, let $\widetilde{\mathcal A}_0$ denote the
connected component of $\widetilde{\mathcal A}$ containing
the identity element. Let
\begin{equation}\label{e11}
{\mathcal G} \,\subset\, \widetilde{\mathcal A}_0
\end{equation}
be a maximal connected reductive subgroup of
$\widetilde{\mathcal A}_0$ \cite[page 217,
Theorem 7.1]{Mo}. (As $\widetilde{\mathcal A}_0$
is connected, from \cite[page 217, Theorem 7.1]{Mo} we
know that any two maximal connected
reductive subgroups of it are conjugate.) Therefore,
the commutator subgroup
$$
{\mathcal G}'\, :=\, [{\mathcal G}\, ,{\mathcal G}]
\,\subset\, {\mathcal G}
$$
is semisimple, where ${\mathcal G}$ is constructed in Eq.
\eqref{e11}. The homomorphism $p$ in Eq. \eqref{e8}
is surjective, and $G$ is simple. Hence the restriction
$$
p'\, :=\, p\vert_{{\mathcal G}'}\, :\, {\mathcal G}'
\, \longrightarrow\, G
$$
is also surjective. Express the semisimple
group ${\mathcal G}'$ as
a quotient of a product of simple and simply
connected groups by a finite group. So
$$
{\mathcal G}'\, =\, (\prod_{i=1}^n G_i)/\Gamma\, ,
$$
where each $G_i$ is a simple and simply
connected linear algebraic group defined over $k$, and
$\Gamma$ is a finite group contained in the center
of $\prod_{i=1}^n G_i$. Let
$$
q'\, :\, \prod_{i=1}^n G_i\, \longrightarrow\, {\mathcal G}'
$$
be the quotient map. Since $G$ is simple and simply
connected, and $p'$ is surjective, there is some
$i_0\, \in\, [1\, ,n]$ such that the homomorphism
$$
p_0\, :=\, (p'\circ q')\vert_{G_{i_0}}\, :\,
G_{i_0}\, \longrightarrow\, G
$$
is an isomorphism.

The homomorphism
$$
\psi\, :=\, q'\circ p^{-1}_0 \, :\,G\,\longrightarrow\,
{\mathcal G}'\, \hookrightarrow\,\widetilde{\mathcal A}
$$
clearly satisfies the splitting condition
$$
p\circ\psi\,=\, \text{Id}_G\, .
$$
Fix a homomorphism $\psi$ as in Eq. \eqref{e10}
such that $p\circ\psi\,=\, \text{Id}_G$.

Now we have an action of $G$ on $E_{H'}$
$$
\rho\, :\, G \times E_{H'}\, \longrightarrow\, E_{H'}
$$
defined by
$$
(g\, ,z)\, \longmapsto\, \psi(g)(z)\, \in\,
(E_{H'})_{\phi(g)(\gamma(z))}\, ,
$$
where $\phi$ and $\psi$ are the maps in Eq. \eqref{e3}
and Eq. \eqref{e10} respectively (the map $\gamma$ is
as in the statement of the
proposition). It is straight--forward to check
that $\rho$ satisfies the two conditions in Definition
\ref{def1}. This completes the proof of the proposition.
\end{proof}

\section{Semistable and polystable homogeneous principal
bundles}

Let $H$ be a connected reductive linear algebraic group
defined over $k$. Let $(E_H\, ,\rho)$ be a homogeneous
principal $H$--bundle over $G/P$.

\begin{lemma}\label{lem1}
The principal $H$--bundle $E_H$ is semistable if and
only if $(E_H\, ,\rho)$ is equivariantly semistable.
\end{lemma}

\begin{proof}
If $E_H$ is semistable, then clearly $(E_H\, ,\rho)$ is
equivariantly semistable. To prove the converse,
assume that $E_H$ is not semistable. Then $E_H$ admits a
unique Harder--Narasimhan reduction
$$
E_Q\, \subset\, E_H
$$
that contradicts the semistability condition of $E_H$
(see \cite[page 211, Theorem 4.1]{BH}). From the
uniqueness of $E_Q$
it follows immediately that the action of $G$
on $E_H$ leaves $E_Q$ invariant. Therefore, $E_H$ is not
equivariantly semistable. This completes the proof of the lemma.
\end{proof}

\begin{lemma}\label{lem2}
Let $(E_H\, ,\rho)$ be a homogeneous
principal $H$--bundle over $G/P$.
The principal $H$--bundle $E_H$ is polystable if and
only if $(E_H\, ,\rho)$ is equivariantly polystable.
\end{lemma}

\begin{proof}
First assume that $E_H$ is polystable. We will show
that $(E_H\, ,
\rho)$ is equivariantly polystable.
Since the characteristic of the field $k$ is zero,
it suffices to prove this under the assumption that
$k\, =\, \mathbb C$. We assume that $k\, =\, \mathbb C$.

Fix a maximal compact subgroup
\begin{equation}\label{K}
K\, \subset\, G\, .
\end{equation}
Fix a K\"ahler form $\omega$ on $G/P$ satisfying the following
two conditions:
\begin{itemize}
\item the action of $K$ on $G/P$ (given by $\phi$ in Eq.
\eqref{e3}) preserves $\omega$, and

\item the cohomology class in $H^2(G/P,\, {\mathbb C})$
represented by the closed form
$\omega$ coincides with $c_1(\xi)$, where
$\xi$ is the fixed ample line bundle on $G/P$.
\end{itemize}

Since the principal $H$--bundle $E_H$ is polystable, it admits
a unique Einstein--Hermitian connection with respect to $\omega$
\cite[page 24, Theorem 1]{RS}, \cite[page 221, Theorem 3.7]{AnB}.
Although the uniqueness of an Einstein--Hermitian connection
is well known, we will explain it here because 
neither of \cite{RS} and \cite{AnB} explicitly mentions it.

On a vector bundle $W$ admitting an Einstein--Hermitian
connection, there is exactly one Einstein--Hermitian
connection. Indeed, if $W$ is indecomposable, then this
is proved in \cite[page 12, Corollary 9 (i)]{Do}; the general
case, where $W$ is a direct sum of indecomposable vector bundles,
follows from this and \cite[page 878, Proposition 3.3]{Si}. Let
$$
Z_0(H)\, \subset\, H
$$
be the connected component, containing
the identity element, of the center of $H$. Take any homomorphism
\begin{equation}\label{beta}
\beta\, :\, H\, \longrightarrow\, \text{GL}(n,{\mathbb C})
\end{equation}
that takes $Z_0(H)$ to the center of
$\text{GL}(n,{\mathbb C})$. Let
$$
E_H(\beta)\, :=\, E_H\times^H {\mathbb C}^n\, \longrightarrow
\, G/P
$$
be the vector bundle associated to the principal $H$--bundle $E_H$
for $\beta$ and the standard representation of $\text{GL}(n,{\mathbb 
C})$. The condition on $\beta$ ensures that the
connection $\nabla(E_H(\beta))$ on $E_H(\beta)$
induced by an Einstein--Hermitian connection $\nabla(E_H)$ 
on $E_H$ is also Einstein--Hermitian. Hence $\nabla(E_H(\beta))$
is the unique Einstein--Hermitian connection
on the vector bundle $E_H(\beta)$.

Fix characters
$$
\chi_i\, :\, H\, \longrightarrow\, {\mathbb C}^*\, ,
$$
$i\, \in\, [1\, ,n]$, such that the map
$$
\prod_{i=1}^n \chi_i\, :\, H/[H\, ,H]\, \longrightarrow\,
({\mathbb C}^*)^n
$$
is an embedding. Let $\mathfrak h$ be the Lie algebra of $H$.
Let $\nabla_1$ and $\nabla_2$ be two connections on
the principal $H$--bundle $E_H$ satisfying the following
conditions:
\begin{itemize}
\item the connections on the adjoint vector bundle
$\text{ad}(E_H)\, :=\, E_H\times^H \mathfrak h$ induced by
$\nabla_1$ and $\nabla_2$ coincide, and

\item for each $i\, \in\, [1\, ,n]$, the connections on the
associated line bundle $E_H\times^{\chi_i} \mathbb C$ induced by
$\nabla_1$ and $\nabla_2$ coincide.
\end{itemize}
Then it is straight forward to check that $\nabla_1$ coincides
with $\nabla_2$.

Now setting the above representations for $\beta$ in
Eq. \eqref{beta} we conclude that $E_H$ admits at
most one Einstein--Hermitian connection.

It should be clarified that although the Einstein--Hermitian
connection is unique, the Einstein--Hermitian metric (which
is a $C^\infty$ reduction of structure group of the principal
$H$--bundle to a maximal compact subgroup of $H$)
is not unique. Any two Einstein--Hermitian reductions on a given
principal $H$--bundle differ by the translation by an element of 
$Z_0(H)/K(Z_0(H))$, where $K(Z_0(H))\, \subset\, Z_0(H)$ is
the maximal compact subgroup.

Let $\nabla(E_H)$ denote the unique Einstein--Hermitian
connection on $E_H$. From the uniqueness of $\nabla(E_H)$ it
follows immediately that the action of the group $K$ in Eq.
\eqref{K} on $E_H$ (given by $\phi$ in Eq.
\eqref{e3}) preserves the connection $\nabla(E_H)$.

In \cite{RS} it is proved that a principal bundle
admitting an Einstein--Hermitian connection is polystable
(see \cite[\S~4, pages 28--29]{RS}). Using the fact that
the action of $K$ on $E_H$ preserves
the Einstein--Hermitian connection $\nabla(E_H)$, this
proof in \cite{RS} gives that $E_H$ is equivariantly polystable
for the action of $K$ on $E_H$. Since $K$ is a maximal compact
subgroup of $G$, any $K$--invariant holomorphic reduction
of structure group of $E_H$ is automatically $G$--invariant.
Therefore, we now conclude that the homogeneous principal
$H$--bundle $(E_H\, ,\rho)$ is equivariantly polystable for the
action of $G$ on $E_H$.

To prove the converse,
assume that $(E_H\, ,\rho)$ is equivariantly polystable. Let
$$
E_{L(Q)}\, \subset\, E_H
$$
be a $G$--invariant minimal Levi reduction of the
structure group \cite[page 56, Theorem 1.3]{BP}. The
action of $G$ on $E_{L(Q)}$ induced by $\rho$ will
also be denoted by $\rho$. We note that the homogeneous
principal $L(Q)$--bundle $(E_{L(Q)}\, ,\rho)$ is
equivariantly stable because $E_{L(Q)}$ is a $G$--invariant
minimal Levi reduction of $E_H$.

Since $(E_{L(Q)}\, ,\rho)$ is
equivariantly stable, from Lemma \ref{lem1} it follows
that the principal $L(Q)$--bundle $E_{L(Q)}$ is semistable.
Let $\text{ad}(E_{L(Q)})$ be the adjoint vector bundle of
$E_{L(Q)}$. We recall that $\text{ad}(E_{L(Q)})$ is the vector
bundle over $G/P$
associated to the principal $L(Q)$--bundle $E_{L(Q)}$ for the
adjoint action of $E_{L(Q)}$ on its own Lie algebra.
The adjoint vector bundle $\text{ad}(E_{L(Q)})$ is semistable
because the principal $L(Q)$--bundle $E_{L(Q)}$ is semistable
\cite[page 285, Theorem 3.18]{RR}. Let
\begin{equation}\label{w0}
W_0\, \subset\, \text{ad}(E_{L(Q)})
\end{equation}
be the socle of the semistable vector bundle
(see \cite[page 23, Lemma 1.5.5]{HL}). Since the vector
bundle $\text{ad}(E_{L(Q)})$ is homogeneous, it follows that
$W_0$ is actually a subbundle of $\text{ad}(E_{L(Q)})$.

We will show that the principal $L(Q)$--bundle $E_{L(Q)}$
is polystable.

To prove this by contradiction, assume that $E_{L(Q)}$
is not polystable. Therefore, the semistable vector bundle
$\text{ad}(E_{L(Q)})$ is not polystable. Consequently, the
subbundle $W_0$ in Eq. \eqref{w0} is a proper one.

In \cite{AnB}, using $W_0$ a unique reduction of structure
group of $E_H$
\begin{equation}\label{tr2.}
E_{Q_0}\, \subset\, E_{L(Q)}
\end{equation}
to a certain proper parabolic subgroup
\begin{equation}\label{proper}
Q_0\,\subsetneq\, L(Q)
\end{equation}
is constructed; the parabolic subgroup
$Q_0$ is also constructed using $W_0$ (see \cite[page 218,
Proposition 2.12]{AnB}). This reduction $E_{Q_0}$ constructed
in \cite[page 218, Proposition 2.12]{AnB} is admissible
(admissible reductions were defined in Section
\ref{sec2}). From the uniqueness of $E_{Q_0}$
in Eq. \eqref{tr2.}
it follows immediately that the action $\rho$ of $G$ on
$E_{L(Q)}$ leaves the subvariety
$E_{Q_0}$ invariant. Since $E_{Q_0}$ is an admissible
reduction of structure group of $E_{L(Q)}$
which is left invariant by the action of $G$ on
$E_{L(Q)}$, and $E_{L(Q)}$ is equivariantly stable, if
follows that $Q_0\, =\, L(Q)$. But this contradicts
Eq. \eqref{proper}.

Consequently, $W_0\, =\, \text{ad}(E_{L(Q)})$. Therefore,
we conclude that the principal $L(Q)$--bundle $E_{L(Q)}$ is
polystable.

Since the principal $L(Q)$--bundle $E_{L(Q)}$ is
polystable, it follows that the principal $H$--bundle $E_H$
is polystable. This completes the proof of the lemma.
\end{proof}

We note that the analog of Lemma \ref{lem1} and Lemma
\ref{lem2} for stable principal bundles is not valid.
In other words, there are equivariantly stable
homogeneous principal $H$--bundles $(E_H\, ,\rho)$ such
that $E_H$ is not stable.

To construct such
an example, take a pair $(H\, ,\eta)$, where
$H$ is a connected reductive nonabelian linear
algebraic group defined over $k$, and
\begin{equation}\label{eta}
\eta\, :\, G\, \longrightarrow\, H
\end{equation}
is a homomorphism satisfying the following condition:
the image $\eta(G)$ is not contained in any proper
parabolic subgroup of $H$. For example, we may take
$H\, =\, G$ and $\eta\, =\, \text{Id}_G$.

Let $E_H$ be the trivial principal $H$--bundle $M\times
H$ over $M\, =\, G/P$. Since $H$ is not abelian, and
$E_H$ is trivial, it follows that the principal
$H$--bundle $E_H$ is not stable.

We will construct an action of $G$ on $E_H$.

Consider the action of $G$ on $G/P$ defined by the
homomorphism $\phi$ in Eq. \eqref{e3} together with the
left translation action of $G$ on $H$ using the
homomorphism $\eta$ in Eq. \eqref{eta}. Now let $\rho$
denote the diagonal action of $G$ on $E_H\,=\, (G/P)\times
H$. We will show that the resulting homogeneous principal
$H$--bundle $(E_H \, ,\rho)$ is equivariantly stable.

To prove by contradiction, assume that
$(E_H \, ,\rho)$ is not equivariantly stable. Since
$E_H$ is trivial, it is polystable. Hence from Lemma
\ref{lem2} we know that $(E_H \, ,\rho)$ is equivariantly
polystable. Therefore, there is a proper parabolic
subgroup $Q\, \subset\, H$ and a $G$--invariant reduction 
of structure group
\begin{equation}\label{eqh}
E_Q\, \subset\, E_H
\end{equation}
over $G/P$ such that
\begin{equation}\label{eqh1}
\text{degree}(\text{ad}(E_H)/\text{ad}(E_Q))\, =\, 0\, ,
\end{equation}
where $\text{ad}(E_H)$ and $\text{ad}(E_Q)$ are the
adjoint vector bundles of $E_H$ and $E_Q$ respectively.

Let $q_0$ denote the dimension of the group $Q$.

Since $E_H\, =\, M\times H$, the adjoint vector bundle
$\text{ad}(E_H)$ is identified with the trivial vector
bundle $M\times {\mathfrak h}$, where $\mathfrak h$ is
the Lie algebra of $H$. Therefore, the subbundle
\begin{equation}\label{th2.}
\text{ad}(E_Q)\, \subset\, \text{ad}(E_H)\, =\,
M\times {\mathfrak h}
\end{equation}
defines a morphism
\begin{equation}\label{theta}
\theta\, :\, G/P\, \longrightarrow\, \text{Gr}(q_0,
{\mathfrak h})\, ,
\end{equation}
where $\text{Gr}(q_0, {\mathfrak h})$ is the Grassmann
variety that parametrizes all subspaces of ${\mathfrak h}$
of dimension $q_0\, :=\,
\dim Q$. The morphism $\theta$ in Eq. \eqref{theta}
sends any $x\, \in\, G/P$ to the subspace
$\text{ad}(E_Q)_x\, \subset\, \text{ad}(E_H)_x\, =\,
\mathfrak h$. Therefore,
$$
\text{ad}(E_H)/\text{ad}(E_Q)\, =\, \theta^*{\mathcal Q}\, ,
$$
where ${\mathcal Q}\, \longrightarrow\, \text{Gr}(q_0,
{\mathfrak h})$ is the tautological quotient bundle (the
fiber of ${\mathcal Q}$ over any point of $\text{Gr}(q_0,
{\mathfrak h})$ is the quotient of $\mathfrak h$ by the
the corresponding subspace). Hence from
Eq. \eqref{eqh1},
\begin{equation}\label{th.}
\text{degree}(\theta^*{\mathcal Q})\, =\,
\text{degree}(\theta^*\det ({\mathcal Q}))\, =\, 0\, ,
\end{equation}
where $\det ({\mathcal Q}) \, =\,
\bigwedge^{\text{top}}{\mathcal Q}$ is the top exterior product
of $\mathcal Q$. The line bundle $\det ({\mathcal Q})$ over
$\text{Gr}(q_0, {\mathfrak h})$ is ample. Hence from Eq.
\eqref{th.} it follows that $\theta$ is a constant map.
Therefore, there is a subspace
\begin{equation}\label{v0}
V_0\, \subset\, {\mathfrak h}
\end{equation}
such that the subbundle $\text{ad}(E_Q)$ in Eq. \eqref{th2.}
coincides with $M\times V_0\, \subset\, M\times {\mathfrak h}$.

Let $\mathfrak q$ be the Lie algebra of $Q$.
Since $\text{ad}(E_Q)$ is the adjoint vector bundle of $E_Q$
it follows that the subspace $V_0$ in Eq. \eqref{v0} is
a conjugate of the subspace ${\mathfrak q}\, \subset\,
{\mathfrak h}$. Therefore, $V_0$ is the Lie algebra of a
parabolic subgroup
\begin{equation}\label{q0}
Q_0\, \subset\, H
\end{equation}
which is a conjugate of $Q$.

Recall the given condition that the action of $G$ on $E_H$
leaves $E_Q$ invariant. This implies that the adjoint action
of $G$ on $\mathfrak h$, defined using the homomorphism $\eta$
in Eq. \eqref{eta}, leaves the subspace $V_0$ invariant.
Since $Q_0$ is a parabolic subgroup (see Eq. \eqref{q0}),
and $V_0$ is the Lie algebra of $Q_0$,
the normalizer of $V_0$ inside $H$ coincides with $Q_0$
(see \cite[page 154, Theorem 11.16]{Bo}, \cite[page 179,
Theorem (c)]{Hu}). Consequently, we have
$$
\eta(G)\, \subset\, Q_0\, .
$$
But this contradicts the given condition that $\eta(G)$ is
not contained in any proper parabolic subgroup of $H$.
Thus there is no $G$--invariant reduction as in Eq.
\eqref{eqh}.

Therefore, we conclude that the homogeneous principal
$H$--bundle $(E_H \, ,\rho)$ is equivariantly stable.

We note that if we set $H\, =\, \text{GL}(n,k)$, then
the above example gives a counter--example to
Corollary 2.11 in \cite{Ro}.

\section{Stable homogeneous principal bundles}

A homogeneous principal $H$--bundle $(E_H\, ,\rho)$ over
$G/P$ is clearly equivariantly stable if $E_H$ is stable.
The following theorem is a converse of this.

\begin{theorem}\label{thm1}
Let $(E_H\, ,\rho)$ be an equivariantly stable
homogeneous principal $H$--bundle over
$G/P$, where $H$ is a connected reductive
linear algebraic group defined over
$k$, such that the principal $H$--bundle
$E_H$ is not stable. Then there is an action
$$
\rho'\, :\, G \times E_H\, \longrightarrow\, E_H
$$
of $G$ on $E_H$ such that the following two hold:
\begin{enumerate}
\item the pair $(E_H\, ,\rho')$ is a homogeneous
principal $H$--bundle, and

\item the homogeneous principal $H$--bundle
$(E_H\, ,\rho')$ is not equivariantly stable.
\end{enumerate}
\end{theorem}

\begin{proof}
Since $(E_H\, ,\rho)$ is equivariantly
stable, from Lemma \ref{lem2} we know that
the principal $H$--bundle $E_H$ is polystable.
Therefore, from the given condition that
$E_H$ is not stable it follows immediately
that $E_H$ admits a reduction of structure group
\begin{equation}\label{re.}
E_{L(Q')}\, \subset\, E_H
\end{equation}
to a Levi subgroup $L(Q')$ of some
proper parabolic subgroup $Q'$ of $H$.

Therefore, there is a natural reduction of structure group
of $E_H$ to a Levi subgroup $L(Q)$ of a parabolic subgroup
$Q\, \subset\, H$
\begin{equation}\label{e12}
E_{L(Q)}\, \subset\, E_H
\end{equation}
which has the following property: the subgroup
$L(Q)\, \subset\, H$ is smallest among
all the Levi subgroups of parabolic subgroups of $H$
to which $E_H$ admits a reduction of structure group
(see \cite[page 230, Theorem 3.2]{BBN1} and
\cite[page 232, Theorem 3.4]{BBN1}). An alternative
construction of this reduction of structure group
in Eq. \eqref{e12} is given in \cite{BBN2}.

Since $L(Q')$ in Eq. \eqref{re.} is a Levi subgroup of a
proper parabolic subgroup of $H$, the Levi
subgroup $L(Q)$ in Eq. \eqref{e12} must be a
proper subgroup of $H$.

It should be mentioned that unlike the two reductions
in Lemma \ref{lem1} and Eq. \eqref{tr2.}, the reduction
in Eq. \eqref{e12} is not unique. However, the isomorphism
class of the principal $L(Q)$--bundle $E_{L(Q)}$ in
Eq. \eqref{e12} is uniquely determined (see \cite[page 232,
Proposition 3.3]{BBN1}). From this it can be deduced
that for each point $g\, \in\, G$, the pulled back principal
$L(Q)$--bundle $\phi(g)^*E_{L(Q)}$ is isomorphic to
$E_{L(Q)}$, where $\phi$ is the homomorphism in Eq.
\eqref{e3}. Indeed, the pulled back
reduction of structure group
$$
\phi(g)^*E_{L(Q)}\, \subset\, \phi(g)^*E_H
$$
is of the type constructed in \cite{BBN1}. The
principal $H$--bundle $\phi(g)^*E_H$ is isomorphic to
$E_H$ because $(E_H\, ,\rho)$ is homogeneous. Hence from
\cite[page 232, Proposition 3.3]{BBN1} it follows
immediately that $\phi(g)^*E_{L(Q)}$ is isomorphic to
the principal $L(Q)$--bundle $E_{L(Q)}$.

Now from Proposition \ref{prop1} we know that there is an
action of $G$ on $E_{L(Q)}$
\begin{equation}\label{e13}
\rho''\, :\, G\times E_{L(Q)}\, \longrightarrow\,E_{L(Q)}
\end{equation}
such that the pair $(E_{L(Q)}\, ,\rho'')$ is a homogeneous
principal $L(Q)$--bundle.

Since $E_{L(Q)}$ is a reduction of structure group of $E_H$,
the action $\rho''$ (see Eq. \eqref{e13}) induces an action
of $G$ on $E_H$. To explain this, first note that
$E_H$ is a quotient of $E_{L(Q)}\times H$. Two points
$(z_1\, ,h_1)$ and $(z_2\, ,h_2)$ of $E_{L(Q)}\times H$ are
identified in the quotient space $E_H$ if and only if there
is an element $g\, \in\, L(Q)$ such that
$(z_2\, ,h_2)\, =\, (z_1g\, ,g^{-1}h_1)$. The action
$\rho''$ of $G$ on $E_{L(Q)}$ and the trivial action of
$G$ on $H$ together define an action of $G$ on $E_{L(Q)}
\times H$. Using the fact that the actions, on $E_{L(Q)}$,
of $L(Q)$ and $G$ (defined by $\rho''$) commute we conclude
that the above action of $G$ on $E_{L(Q)}\times H$ descends
to an action of $G$ on the quotient space $E_H$. Let
$$
\rho'\, :\, G\times E_H\, \longrightarrow\,E_H
$$
be this descended action. It is now easy to see
that this pair $(E_H\, ,\rho')$ is a homogeneous principal
$H$--bundle.

The action $\rho'$ of $G$ on $E_H$ preserves the subvariety
$$
E_{L(Q)}\, \subset\, E_H
$$
in Eq. \eqref{e12}. In fact, the restriction of $\rho'$ to
$E_{L(Q)}$ coincides with $\rho''$. We noted earlier that
$L(Q)$ is a proper subgroup of $H$. Hence the existence
of the $\rho'$ invariant reduction $E_{L(Q)}\,\subset\,
E_H$ proves that $(E_H\, ,\rho')$ is not equivariantly
stable. This completes the proof of the theorem.
\end{proof}


\end{document}